\documentclass[a4paper,11pt]{article}
\usepackage[a4paper,left=26.5mm,right=26mm,top=26.25mm,bottom=36.75mm]{geometry}

\usepackage[english]{babel}	
\usepackage{lmodern}
\usepackage[T1]{fontenc}
\usepackage[utf8]{inputenc}	

\usepackage{graphicx}      
\usepackage[numbers]{natbib}        

\usepackage{graphicx}
\usepackage{enumitem}
\usepackage{array}
\usepackage{color}
\usepackage{stmaryrd}
\usepackage{amsmath}
\usepackage{amssymb}
\usepackage{dsfont}
\usepackage{mathrsfs}
\usepackage{amsthm}
\usepackage{mathtools}
\usepackage{pdflscape}
\usepackage{authblk} 

\usepackage{bbm}

\renewcommand{\emptyset}{\varnothing}

\DeclareMathAlphabet{\mathbbold}{U}{bbold}{m}{n}
\newcommand{\Id}{I}

\newcommand{\R}{\mathbb R}

\newcommand*\diff{\mathop{}\!\mathrm{d}}

\newcommand{\xhat}{\hat{x}}
\newcommand{\Xhat}{\hat{X}}
\newcommand{\eps}{\varepsilon}
\newcommand{\Eps}{E}

\newcommand{\D}{\mathcal{D}}

\renewcommand{\L}{\mathcal{L}}

\renewcommand{\le}{\leqslant}

\renewcommand{\leq}{\leqslant}
\renewcommand{\geq}{\geqslant}

\newcommand{\HH}{\mathcal{H}}

\setlist[enumerate]{align=left}

\newtheorem{thm}{Theorem}
\newtheorem{cor}[thm]{Corollary}
\newtheorem{lem}[thm]{Lemma}

\theoremstyle{definition}

\newtheorem{assum}[thm]{Assumption}
\newtheorem{exmp}[thm]{Example}
\newtheorem{rem}[thm]{Remark}

\title{Dynamic Output Feedback Stabilization of Non-uniformly Observable Dissipative Systems}

\date{\today}

\author[1]{Ludovic Sacchelli}
\author[2]{Lucas Brivadis}
\author[2]{Vincent Andrieu}
\author[2]{Ulysse Serres}
\author[3]{Jean-Paul Gauthier}

\affil[1]{Department of Mathematics, Lehigh University, Bethlehem, PA, USA (lus219@lehigh.edu)}
\affil[2]{Univ. Lyon, Universit\'e Claude Bernard Lyon 1, CNRS, LAGEPP UMR 5007, 43 bd du 11 novembre 1918, F-69100 Villeurbanne, France (lucas.brivadis@univ-lyon1.fr).}
\affil[3]{Universit\'e de Toulon, Aix Marseille Univ, CNRS, LIS, France
}

\begin{document}

\maketitle

\begin{abstract}
Output feedback stabilization of control systems is a crucial issue in engineering.
Most of these systems are not uniformly observable, which proves to be a difficulty to move from state feedback stabilization to dynamic output feedback stabilization.
In this paper, we present a methodology to overcome this challenge in the case of dissipative systems by requiring only target detectability. These systems appear in many physical systems and we provide various examples and applications of the result.

\end{abstract}

{\small
Keywords: State observers, Dynamic output feedback, Nonlinear control systems, Unobservable, Feedback stabilization, Asymptotic stability
}


\section{Introduction}

A control system is said to be \emph{state-affine} if it is of the form
\begin{equation}\label{E:system}
    \dot x = A(u)x + B(u)
\end{equation}
with $x \in \R^n$ being the state of the system and $u \in \R^m$ being the input.
System \eqref{E:system} is said to be uniformly \emph{dissipative} if there exists a positive definite matrix $P \in \R^{n \times n}$ such that for all inputs $u \in \mathcal{U}\subset\R^m$,
\begin{align}\label{eq:diss}
    PA(u)+A(u)'P\leq0.
\end{align}
Many physical systems satisfy such a dissipativity property.
For example, it is the case for input-state-output port-Hamiltonian systems (see \emph{e.g.} \cite{van2014port}).
Hence, stabilizing the state of such a dissipative state-affine control system to a target point (which we assume to be $0$, without loss of generality) is a crucial issue in engineering.
However, in most cases, only part of the system is measured, namely the output $y$ which we assume to depend linearly on $x$:
\begin{equation}
    y = Cx.
\end{equation}
Hence the stabilization must be achieved using only an output feedback, and not a state feedback.
If a stabilizing state feedback exists, then a common strategy to build a dynamic output feedback is to apply this feedback to an \emph{observer} of the state, which learns the state from the output.
A sufficient condition given in \cite{TeelPraly1994} for this method to work is the \emph{uniform observability} of the system in small time, that is, for all input $u$, the observation of the output $y$ on any time interval is sufficient to uniquely determine the state $x$.

However, as shown in \cite{Gauthier_book}, it is generic for a state-affine system to have inputs that make the system unobservable, called singular inputs.
%
This proves to be a challenge in the context of dynamic output feedback stabilization.
To overcome this issue, some authors have proposed to modify the feedback law in order to get new observability properties while maintaining its stabilizing property (see \emph{e.g.} \cite{Coron1994,ShimTeel2002} for time-varying perturbations or \cite{nous,MarcAurele} for smooth autonomous perturbations).

In this paper, we show for uniformly dissipative systems that local asymptotic state feedback stabilizability and target detectability are sufficient to prove semi-global asymptotic dynamic output feedback stabilizability.
The key point is that target detectability is a much weaker assumption than uniform observability. And we do not follow any perturbation strategy of the feedback law.

In the next section, we state the main result of the paper.
Its proof may be found in Section~\ref{sec:proof}.
In the last Section~\ref{sec:appli}, we provide various examples and applications of the result.

\section{Main result}

\subsection{Problem statement}\label{sec:prob}


Let $n$, $m$ and $p$ be positive integers, $A:\R^m\to\R^{n\times n}$ and $B:\R^m\to\R^{n}$ be two continuous maps, and $C\in\R^{p\times n}$.
For all $u\in C^0(\R_+,\R^m)$, we consider the following observation-control system:
\begin{equation}\label{E:observation_system}
\left\{
\begin{aligned}
&\dot{x}= A(u)x+B(u)
\\
&y= C x
\end{aligned}
\right.
\end{equation}
where $x$ is the state of the system, $u$ is the input and $y$ is the output.

Then \eqref{E:observation_system} is said to be dissipative for the input $u$ if there exists a positive definite matrix $P\in\R^{n\times n}$ such that \eqref{eq:diss} holds, \emph{i.e.} for all $x\in\R^n$,
$
x' P A(u) x \leq 0.
$
Then, for all $\alpha>0$, consider the following Luenberger observer
\begin{equation}\label{E:observer_u}
\left\{
\begin{aligned}
&\dot{\xhat}= A(u)\xhat
+ B(u)
- \alpha P^{-1}C'C\eps\\
&\dot{\eps}= \left(A(u) - \alpha P^{-1}C'C\right)\eps
\end{aligned}
\right.
\end{equation}
where $\xhat$ denotes the state of the observer and $\eps$ the error between the actual state of \eqref{E:observation_system} and the observer, \emph{i.e.} $\eps=\xhat-x$.
Assuming the local asymptotic stabilizability at $0$ of \eqref{E:observation_system},
we prove the semi-global asymptotic stabilizability at $(0, 0)$ of \eqref{E:observer_u},
without uniform observability hypothesis. Instead, we use a detectability hypothesis at the target.

\subsection{Statement of the result}\label{sec:main}

Let $\lambda\in C^0(\R^n, \R^m)$ be a state feedback law. In the rest of the paper, $\lambda$ is fixed and we assume that for all $\alpha>0$, all positive definite $P\in\R^{n\times n}$ and all initial conditions $(\xhat_0, \eps_0)\in\R^n\times\R^n$ the Luenberger observer
\begin{equation}\label{E:observer_system}
\left\{
\begin{aligned}
&\dot{\xhat}= A(\lambda(\xhat))\xhat
+ B(\lambda(\xhat))
- \alpha P^{-1}C'C\eps\\
&\dot{\eps}= \left(A(\lambda(\xhat)) - \alpha P^{-1}C'C\right)\eps.
\end{aligned}
\right.
\end{equation}
has a unique solution $(\xhat, \eps)\in C^0(\R_+, \R^n\times\R^n)$
such that $(\xhat(0), \eps(0)) = (\xhat_0, \eps_0)$ (which is the case when the vector field associated to \eqref{E:observer_system} is locally Lipschitz and forward complete).
Without loss of generality on the system, 
we assume that $\lambda(0)=0$.
In order to answer the issue raised in Section 2, we consider the case of $\lambda$ a locally asymptotically stabilizing state feedback law  of 
\eqref{E:observation_system}.

\begin{assum}[State feedback stabilizability]\label{assum:stab}
There exists a non-empty open domain $\D\subset\R^n$ such that
$0$ is an asymptotically stable equilibrium point with basin of attraction $\D$ of the vector field $f:\R^n\ni x\mapsto A(\lambda(x))x+B(\lambda(x))$.
\end{assum}
As stated earlier, we focus on the class of dissipative systems.
The stabilizing feedback law being fixed, it is sufficient to assume
uniform
dissipativity for inputs that this feedback may yield.
\begin{assum}[Dissipativity]\label{assum:diss}
There exists a positive definite  matrix $P\in\R^{n\times n}$ such that for all
$(x, \xhat)\in\R^n\times\D$
    \begin{equation}
        x' P A(\lambda(\xhat)) x \leq 0.
    \end{equation}
\end{assum}
From now on, \eqref{E:observer_system} is considered only for this specific matrix $P$.
Our goal is to prove the semi-global asymptotic stability of \eqref{E:observer_system} without uniform observability assumption.
We show that the following condition is sufficient.
\begin{assum}[Target detectability]\label{assum:obs}
The pair $(C, A(0))$ is detectable.
\end{assum}
Now we state the main theorem of the paper.
\begin{thm}\label{thmain}
If Assumptions~\ref{assum:stab}, \ref{assum:diss} and \ref{assum:obs} hold,
then for any compact set $K_1\times K_2\subset \D\times \R^n$, there exists $\alpha_0>0$ such that for all $\alpha\in(0, \alpha_0)$, $(0, 0)$ is an asymptotically stable equilibrium point of \eqref{E:observer_system} with basin of attraction containing $K_1\times K_2$.

\end{thm}

\subsection{Discussion on the result}

    \begin{rem}[on Assumption~\ref{assum:diss}]
    Dissipativity of the system is the key point of the result. It implies that the function $V:\eps\mapsto\eps'P\eps$ is a Lyapunov function for the $\eps$-subsystem of \eqref{E:observer_system} as long as $\xhat\in\D$. Indeed,
    \begin{align*}
    \frac{\diff W(\eps)}{\diff t}
    &= \eps'P\dot\eps + \dot\eps'P\eps\nonumber\\
    &= \eps'\left(P A(\lambda(\xhat)) + A(\lambda(\xhat))'P\right) \eps - 2\alpha \eps'C'C\eps\nonumber\\
    &\leq - 2\alpha |C\eps|^2 \tag{by Assumption~\ref{assum:diss}}\\
    &\leq0.
    \end{align*}
    Numerous physical systems satisfy this dissipativity property.
    We provide examples in the last Section~\ref{sec:appli}.
    
    \end{rem}
    
    \begin{rem}[on Assumption~\ref{assum:obs}]
    Let $A_0=A(0)$.
    It is well-known (see \emph{e.g} \cite{sontag}) that the detectability of the pair $(C, A_0)$ is equivalent to the fact that the solutions of the linear dynamical system $\dot\omega=A_0\omega$ are such that, if $C\omega\equiv0$, then $\omega\to0$.
    Note that the set of pairs $(C, A_0)$ that are detectable is open and dense in $\R^{m\times n}\times\R^{n\times n}$.
    \end{rem}

\begin{rem}[on Theorem~\ref{thmain}]
The key point of the theorem is that it does not rely on any \emph{uniform observability} assumption.
Instead, a dissipativity property of the system is required, as well as target detectability.

\end{rem}
Note that taking $p=n$ and $C$ the identity matrix, Assumption~\ref{assum:obs} is trivially satisfied, and then Theorem~\ref{thmain} implies the following corollary, which is an interesting result in itself about the stabilization of dissipative systems.
\begin{cor}
Any locally asymptotically stabilizable dissipative state-affine system is also semi-globally asymptotically stabilizable by means of a dynamic feedback.
\end{cor}

We prove Theorem~\ref{thmain} in the next section.

\section{Proof}\label{sec:proof}

In this section, we suppose that Assumptions~\ref{assum:stab}, \ref{assum:diss} and \ref{assum:obs} hold.
The proof of Theorem~\ref{thmain} rely on the three following lemmas, that we state and prove in this section.

\subsection{Local asymptotic stability}\label{secloc}

\begin{lem}\label{lem:loc}
For all $\alpha>0$, \eqref{E:observer_system} is locally asymptotically stable at $(0, 0)$.
\end{lem}
\begin{proof}
Let $A_0 = A(0)$.
Consider the linearization of \eqref{E:observer_system} at the origin:
\begin{equation}\label{E:linearized_observer_system}
\left\{
\begin{aligned}
&\dot{\xhat}= A_0\xhat - \alpha P^{-1}C'C\eps\\
&\dot{\eps}= \left(A_0 - \alpha P^{-1}C'C\right)\eps.
\end{aligned}
\right.  
\end{equation}
This system is upper triangular. Let us first focus on the $\eps$ part of the system. Consider the function $V:\eps \mapsto \eps'P\eps$. Then $V$ is a Lyapunov function for the $\eps$-subsystem. Indeed,
\begin{align}
    \frac{\diff V(\eps)}{\diff t}
    &= \eps'P\dot\eps + \dot\eps'P\eps\nonumber\\
    &= \eps'\left(P A_0 + A_0'P\right) \eps - 2\alpha \eps'C'C\eps\nonumber\\
    &\leq - 2\alpha |C\eps|^2\tag{by Assumption~\ref{assum:diss}}\\
    &\leq 0.\nonumber
\end{align}
We denote by $\omega(\eps_0)$ the $\omega$-limit set of the the $\eps$-subsystem with initial condition $\eps_0\in\R^n$.
Then, by LaSalle's invariance principle, $\omega(\eps_0) \subset \{\eps_0\in\R^n\mid C\eps\equiv 0\}$. Since the pair $(C, A_0)$ is detectable by Assumption~\ref{assum:obs}, we have $\eps\to0$. Since the system is linear, this implies that all eigenvalues of $A_0-\alpha P^{-1}C'C$ have negative real part.
Now let us consider the $\xhat$-subsystem.
Since $0$ is asymptotically stable for the vector field $f$, all the eigenvalues of $A_0$ have non-positive real part.
Moreover, $\{\eps_0\in\R^n\mid C\eps\equiv 0\}$ is invariant under the dynamics of the $\xhat$-subsystem.
Then, applying the center manifold theorem (see \emph{e.g.} \cite[Appendix, Theorem 4.2]{Gauthier_book}), \eqref{E:observer_system} is locally asymptotically stable at $0$.
\end{proof}

\subsection{All trajectories are bounded}

\begin{lem}\label{lem:bound}
For all compact set $K_1\times K_2\subset\D\times\R^n$, there exists $\alpha_0>0$ such that for all $\alpha\in(0, \alpha_0)$, all the trajectories of \eqref{E:observer_system} with initial conditions in $K_1\times K_2$ remain in a compact subset of $\D\times\R^n$.
\end{lem}

\begin{proof}
Let $K_1\times K_2\subset\D\times\R^n$ be a compact set.
Let $R = \mu_{\max}\sup_{K_2}V < +\infty$, where $\mu_{\max}$ denotes the largest eigenvalue of $P$.
According to Assumption~\ref{assum:stab}, there exists a strict Lyapunov function $W$ for the vector field $f$.
For all $r>0$, set $D(r) = \{x\in\R^n\mid W(x) < r\} \subset\D$ and denote by $\partial D(r)$ its boundary.
Let $\rho>0$ be such that $K_1\subset D(\rho)$
and the closure of $D(\rho)$ lies in $\D$.
Set $M_1 = \sup_{\partial D(\rho)} L_f W < 0$
and $M_2 = 1+ \sup_{\partial D(\rho)} |\nabla W| < +\infty$ where $L_f$ denotes the usual Lie derivative along $f$ and $\nabla$ stands for the Euclidean gradient.

Let $\alpha_0 = \frac{-M_1}{RM_2|P^{-1}||C|^2}>0$ and take $\alpha\in(0,\alpha_0)$.
Take $(\xhat_0, \eps_0)\in K_1\times K_2$ and denote $(\xhat, \eps)$ the semi-trajectory of \eqref{E:observer_system} starting from $(\xhat_0, \eps_0)$.
Since $V:\eps\mapsto\eps'P\eps$ is a Lyapunov function for the $\eps$-subsystem of \eqref{E:observer_system}, we have $|\eps|\leq R$. Assume there exists $t>0$ such that $W(\xhat(t))=\rho$.
Then
\begin{align*}
    \frac{\diff}{\diff t}W(\xhat(t))
    &= L_f W (\xhat(t)) - \alpha
    \left(\nabla W(\xhat(t))\right)'
    P^{-1}C'C\eps(t)\\
    &\leq M_1 + \alpha M_2 |P^{-1}||C|^2 R \\ 
    &<0.
\end{align*}
Hence $\xhat(t)\in D(\rho)$ for all $t\geq0$.
Thus, for all $\alpha\in(0,\alpha_0)$, all the trajectories of \eqref{E:observer_system} with initial conditions in $K_1\times K_2$ remain in a compact subset of $\D\times\R^n$.

\end{proof}

\subsection{All trajectories converge to $0$.}\label{secconv}

\begin{lem}\label{lem:conv}
For all compact set $K_1\times K_2\subset\D\times\R^n$, there exists $\alpha_0>0$ such that for all $\alpha\in(0, \alpha_0)$, all the trajectories of \eqref{E:observer_system} with initial conditions in $K_1\times K_2$ converge to $(0, 0)$.

\end{lem}

\begin{proof}
For all initial conditions $(\xhat_0, \eps_0)\in\R^{2n}$, we denote $t\mapsto(\Xhat(t, \xhat_0, \eps_0), \Eps(t, \xhat_0, \eps_0))$ the semi-trajectory of \eqref{E:observer_system} with initial conditions $(\xhat_0, \eps_0)$.
Let $K_1\times K_2\subset\D\times\R^n$ be a compact set and $\alpha>\alpha_0>0$ as in Lemma~\ref{lem:bound}, so that all the trajectories of \eqref{E:observer_system} with initial conditions in $K_1\times K_2$ remain in a compact subset of $\D\times\R^n$.
Let $(\xhat_0, \eps_0)\in K_1\times K_2$.
We denote $(\xhat, \eps)$ the semi-trajectory of \eqref{E:observer_system} starting from $(\xhat_0, \eps_0)$, and $\omega(\xhat_0, \eps_0)$ the $\omega$-limit set of this semi-trajectory.

We  prove that $(\xhat, \eps)$ converges to $(0,0)$ as a consequence of Lemma~\ref{lem:loc}, by proving that the semi-trajectory enters the basin of attraction of $(0,0)$ in finite time.
It is sufficient to prove that $(0,0)\in \omega(\xhat_0, \eps_0)$ since this implies that $(\xhat, \eps)$ enters any open set containing $(0,0)$ in finite time.
We prove this in three steps: first $\omega(\xhat_0, \eps_0)\subset\{(\xhat_1, \eps_1)\in\D\times\R^n\mid C\Eps(\cdot, \xhat_1, \eps_1)\equiv0\}$, then $\omega(\xhat_0, \eps_0)\cap \left(\{0\}\times\R^n\right)\neq \emptyset$ and finally $(0,0)\in \omega(\xhat_0, \eps_0)$.
Recall that
$
\frac{\diff V(\eps)}{\diff t}
\leq - 2\alpha |C\eps|^2
$
by Assumption~\ref{assum:diss}.
Then, according to LaSalle's invariance principle, $\omega(\xhat_0, \eps_0)\subset \{(\xhat_1, \eps_1)\in\D\times\R^n\mid C\Eps(\cdot, \xhat_1, \eps_1)\equiv0\}$.

Let $(\xhat_1, \eps_1)\in\omega(\xhat_0, \eps_0)$.

The set $\omega(\xhat_0,\eps_0)$ is compact and invariant under the dynamics of the system, hence $\Xhat(t, \xhat_1, \eps_1)\in \omega(\xhat_0,\eps_0)$ for all $t\geq0$.
This further implies that  $\omega(\xhat_1, \eps_1)$ is a non-empty compact subset of $\omega(\xhat_0, \eps_0)$. 
As a consequence of  Assumption~\ref{assum:stab}, $\Xhat(t, \xhat_1, \eps_1) \to 0$ as $t\to+\infty$. 
Hence $\omega(\xhat_1, \eps_1) \subset \{0\}\times\R^n$ and thus
$$
\omega(\xhat_0, \eps_0)\cap \left(\{0\}\times\R^n\right)\neq \emptyset.
$$

Then there exists $\eps_2\in\R^n$
such that $(0,\eps_2)\in \omega(\xhat_0, \eps_0)\subset \{(\xhat_1, \eps_1)\in\D\times\R^n\mid C\Eps(\cdot, \xhat_1, \eps_1)\equiv0\}$.
Hence $\Xhat(\cdot, 0, \eps_2)\equiv0$.
Then $\Eps(\cdot, 0, \eps_2)$ is solution of
\begin{equation}
\dot\eps = A_0\eps,\qquad
C\eps= 0.
\end{equation}
Since the pair $(C, A_0)$ is detectable (by Assumption~\ref{assum:obs}), $\Eps(\cdot, 0, \eps_2)\to0$. Hence $\{(0, 0)\}=\omega(0, \eps_2)\subset \omega(\xhat_0, \eps_0)$.
By local asymptotic stability of $(0,0)$, it follows that the semi-trajectory $(\xhat, \eps)$ converges towards $0$.
\end{proof}

\begin{proof}(Proof of Theorem~\ref{thmain})
Combining stability from Lemma~\ref{lem:loc} and semi-global convergence towards $(0,0)$ from Lemma~\ref{lem:conv}, we get the result.
\end{proof}

\section{Examples and applications}\label{sec:appli}

In this section, we provide some examples and applications to illustrate the main Theorem~\ref{thmain}.

\begin{rem}
If $A(u) = (J(u)-R(u))\HH$ for some positive definite matrix $\HH$ and positive semi-definite (resp. skew-symmetric) matrix $R(u)$ (resp. $J(u)$), $B$ is linear and $C = B'\HH$, then we recognize an input-state-output port-Hamiltonian system (see \emph{e.g.} \cite{van2014port}). In that case, a static output stabilizing feedback is given by $u= -ky$ for any $k>0$.
However, for the same dynamics with a different linear output (\emph{i.e.} such that $C\neq B'\HH$), our result provides a methodology for semi-global dynamic output feedback stabilization when the pair $(C, A(0))$ is detectable. The following examples are of this form.
\end{rem}

\begin{exmp}[Harmonic oscillator]
Consider \eqref{E:observation_system} with
$$A(u) = \begin{pmatrix}0&-(1+u)\\1+u&0\end{pmatrix},\ B(u) = \begin{pmatrix}u\\0\end{pmatrix}\text{ and }C = \begin{pmatrix}0&1\end{pmatrix}.$$
Let $\lambda:\R^2\ni(x_1, x_2)\mapsto -x_1$. Then $W:\R^2\ni x\mapsto|x|^2$ is a Lyapunov function for the vector field $f:x\mapsto A(\lambda(x))x+B(\lambda)$.
Indeed, for any solution $x$ of \eqref{E:observation_system},
\begin{align*}
    \frac{\diff W(x)}{\diff t}
    = 2x'A(\lambda(x))x + 2x'B(\lambda(x)) = -2x_1^2
\end{align*}
since $A(u)$ is skew-symmetric for all $u\in\R$.
According to the LaSalle's invariance principle, the $\omega$-limit set of the trajectory is the largest positively invariant set contained in $\{x\in\R^2\mid x_1\equiv0\}$. Note that $\lambda\equiv0$ and $\dot x_1 = -x_2$ on this set. Then $x\to0$.
Hence $\lambda$ is a globally asymptotically stabilizing feedback law.
The Kalman observability matrix of the pair $(C, A(0))$ is the full rank matrix
\begin{align*}
\begin{pmatrix}
C\\CA(0)
\end{pmatrix}
=\begin{pmatrix}
0&1\\1&0
\end{pmatrix}.
\end{align*}
Hence $(C, A(0))$ is observable, and \emph{a fortiori} detectable.
Thus, all the Assumptions~\ref{assum:stab}, \ref{assum:diss} and \ref{assum:obs} are satisfied, and we may apply Theorem~\ref{thmain} to find a semi-globally asymptotically stabilizing dynamic output feedback:
for all compact set $K_1\times K_2\subset \R^n\times \R^n$, there exists $\alpha_0>0$ such that for all $\alpha\in(0, \alpha_0)$, $(0, 0)$ is an asymptotically stable equilibrium point with basin of attraction containing $K_1\times K_2$ of \eqref{E:observer_system}.
\end{exmp}

\begin{exmp}[\'Cuk converter]\label{ex:cuk}
The averaged model of the \'Cuk converter given in Figure~\ref{fig:circuit} can be written as follow:
\begin{align}\label{E:circuit}
\dot x
=
\begin{pmatrix}0&-(1-u)&0&0\\1-u&0&u&0\\0&-u&0&-1\\0&0&1&-\frac{1}{R}\end{pmatrix}
P
x
+
\begin{pmatrix}E\\0\\0\\0\end{pmatrix}
\end{align}
where $x_1$ and $x_3$ are the fluxes in the inductances $L_1$ and $L_3$, $x_2$ and $x_4$ are the charges in the capacitors $C_2$ and $C_4$, $R$ is the load resistance, $E$ is the voltage source and $P = \mathrm{diag}\begin{pmatrix}\frac{1}{L_1},&\frac{1}{C_2},&\frac{1}{L_3},&
\frac{1}{C_4}\end{pmatrix}$. As in \cite{Rodriguez}, the goal is to stabilize the system at
\begin{align*}
    x^* = \begin{pmatrix}
    \displaystyle\frac{L_1}{REV_d^2},&\displaystyle
    C_2V_d + E,&\displaystyle
    -\frac{L_3}{R}V_d,&\displaystyle
    -C_4V_d
    \end{pmatrix}'
\end{align*}
for some output capacitor voltage $V_d$, which is attained for $u^* = \frac{V_d}{V_d+E}$.
Assume that only the charge $x_2$ is measured,
and address the problem of output feedback stabilization.
In order to match \eqref{E:circuit} and \eqref{E:observation_system}, we set $\bar{x} = x-x^*$ and $\bar{u} = u-u^*$. Then \eqref{E:circuit} can be rewritten as \eqref{E:observation_system} by replacing $x$ by $\bar{x}$ and $u$ by $\bar{u}$ and with
\begin{align*}
    &A(\bar{u}) = \begin{pmatrix}0&-(1-u^*-\bar{u})&0&0\\1-u^*-\bar{u}&0&u^*+\bar{u}&0\\0&-u^*-\bar{u}&0&-1\\0&0&1&-\frac{1}{R}\end{pmatrix}P,\\
    &B(\bar{u}) = \bar{u}b\text{ with } b = \begin{pmatrix}C_2x^*_2\\L_3x^*_3-L_1x^*_1\\-C_2x^*_2\\0\end{pmatrix}
    \ \text{and}\
    C = \begin{pmatrix}0,&1,&0,&0\end{pmatrix}.
\end{align*}
Remark that $u \equiv 1$ and $u \equiv 0$ renders \eqref{E:observation_system} unobservable, since the Kalman observability matrices of the pairs $(C, A(1-u^*))$ and $(C, A(-u^*))$ are not invertible. So the well-known results for dynamic output feedback stabilization of uniformly observable systems do not apply. Theorem~\ref{thmain} may overcome this difficulty.
It remains to check Assumptions~\ref{assum:stab}, \ref{assum:diss} and \ref{assum:obs}.

The system is dissipative since $PA(\bar{u}) + A(\bar{u})'P$ is negative semi-definite for all input $\bar{u}$.
The pair $(C, A(0))$ is observable, and \emph{a fortiori} detectable, since its Kalman observability matrix is full rank as soon as $u^*\neq1$ and $u^*\neq0$ \emph{i.e.} $E\neq0$ and $V_d\neq0$.
Consider the saturated feedback law $\lambda(\bar{x}) = \mathrm{sat}\left(-\beta b'P\bar{x}\right)$, where $\beta>0$ is a tuning parameter and $\mathrm{sat}$ is a saturation function such that $u^* + \lambda$ lies in $(0, 1)$, which is always possible since $u^*\in(0, 1)$.
Then $x\mapsto x'Px$ is a Lyapunov function of the vector field $f:x\mapsto A(\lambda(x))x+B(\lambda(x))$, and according to the LaSalle's invariance principle, the $\omega$-limit set of any trajectory is the largest positively invariant set contained in $\{x\in\R^2\mid b'Px\equiv0\}$, which gives $x\to0$ when $(b'P, A(0))$ is observable.
Hence, for almost all choice of parameters, $\lambda$ is a globally asymptotically stabilizing feedback law.
One may also choose any other locally asymptotically stabilizing feedback law, for example the one given in \cite{Rodriguez}.

Then, Theorem~\ref{thmain} applies, and \eqref{E:observer_system} gives a semi-globally asymptotically stabilizing dynamic output feedback.
In Figures~\ref{fig:simucuk} and \ref{fig:obscuk}, we provide numerical simulations for the following choice of parameters (as in \cite{Rodriguez}):
\begin{table}[h!]
\caption{Numerical values for the simulation of the \'Cuk converter}
\begin{center}
\begin{tabular}{|c|c|c|c|c}
\hline
$L_1$ & $C_2$ & $L_3$ & $C_4$\\
$10.9\,$mH & $22.0\,\mu$F & $10.9\,$mH & $22.9\,\mu$F\\
\hline
$R$ & $E$ & $V_d$ & $\beta$\\
$22.36\,\Omega$ & $12\,$V & $25\,$V & $10^{-4}$\\
\hline
\end{tabular}
\end{center}
\end{table}

For these values, the pair $(b', A(0))$ is observable, hence $\lambda$ is a stabilizing state feedback law.
We choose the initial conditions $x(0) = 0$ and $\xhat(0) = x^*$.
In Figure~\ref{fig:simucuk}, we plot the output voltage $\frac{x_4}{C_4}$ that we want to stabilize at $V_d$ for the state feedback law $\lambda$ and for the dynamic output feedback based on the Luenberger observer for $\alpha  = 1$, $\alpha=10$ and $\alpha=100$.
In Figure~\ref{fig:obscuk}, we plot the error between the actual state of the system and the observer for the same values of $\alpha$.
When $\alpha$ is larger, the observer converges faster to the state of the system.
For $\alpha=100$, $\xhat$ converges quickly to $x$, and then the dynamics of $x$ obtained via the dynamic output feedback is close to the one obtained via state feedback.
On the contrary $\alpha=1$ leads to a slow convergence of the observer. Then, the state dynamics is very close to the one with the constant control $u\equiv\lambda(\xhat(0))=u^*$
Finally, $\alpha = 10$ is a compromise between these two behaviours: the state dynamics is similar to the case where $\alpha=1$ at the beginning, and to the case where $\alpha=100$ at the end of the simulation.

\end{exmp}

\begin{figure}[h!]
\begin{center}
\includegraphics{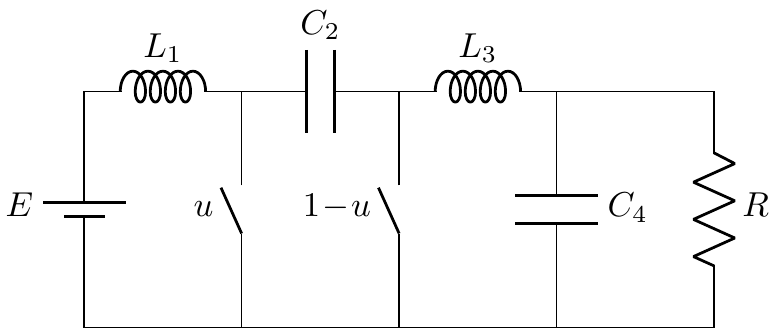}
\caption{Ideal \'Cuk converter.}
\label{fig:circuit}
\end{center}
\end{figure}
\begin{figure}
    \centering
    \includegraphics[width =0.6\linewidth]{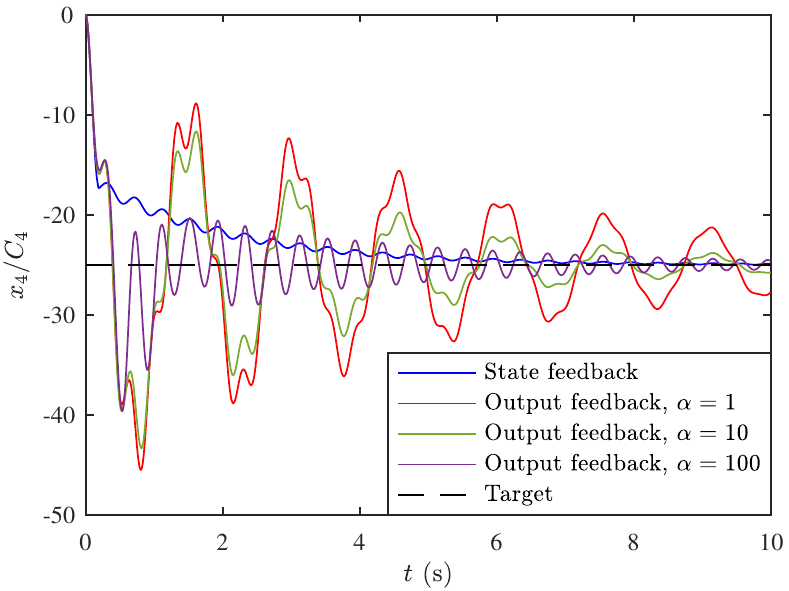}
    \caption{Output voltage of the \'Cuk converter with the state feedback law $\lambda$ and with the corresponding dynamic output feedback law based on the Luenberger observer for different values of $\alpha$.}
    \label{fig:simucuk}
\end{figure}
\begin{figure}
    \centering
    \includegraphics[width =0.6\linewidth]{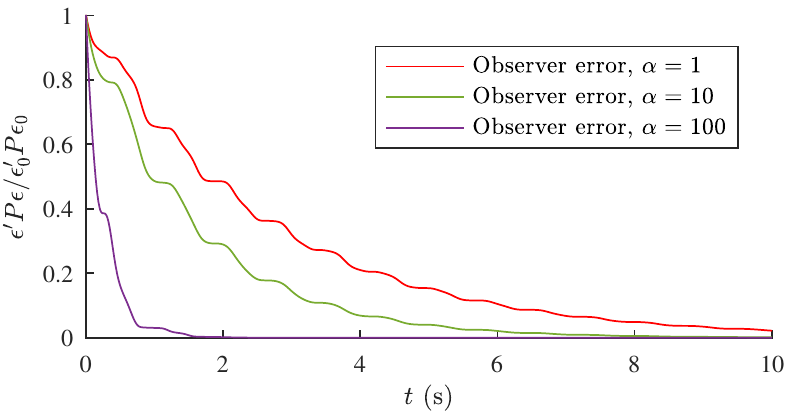}
    \caption{Evolution of the error between the actual state of the \'Cuk converter and the observer for different values of $\alpha$.}
    \label{fig:obscuk}
\end{figure}

\begin{rem}
The matrix $A(0)$ is Hurwitz for any $u^*\in(0, 1)$.
Hence, the constant control $\bar{u} = 0$ \emph{i.e.} $u=u^*$ stabilizes the system at the target point. This phenomenon is due to the load resistance $R$.
However, the user does not have any control on $R$, so this strategy potentially leads to a very slow stabilization.
Indeed, taking $R\to+\infty$ or $R\to0$, some eigenvalues of $A(0)$ converge to the imaginary axis.
In this case, the damping assignment state feedback is much more efficient, and that is why we build a dynamic output feedback based on this state feedback.
A similar remark holds for the next example.
\end{rem}


\begin{exmp}[Heat exchanger]
In \cite{zitte} (which we refer reader to for details), a model of a counter-current heat exchanger is introduced.
The system is 6-dimensional, and each component $x_i$ of the state represents the temperature of one exchanger's compartment.
After a change of coordinates and control (as in the previous Example~\ref{ex:cuk}),
the system can be rewritten in form of \eqref{E:observation_system} with
\begin{align*}
    &A(\bar{u}) = \begin{pmatrix}
    -k\Id_3+\gamma_1(u^*+\bar{u})J&k\Id_3\\
    k\Id_3&-k\Id_3+\gamma_2J'
    \end{pmatrix},\
    B(\bar{u}) = \bar{u}b&
    \\
&\text{with } b = \begin{pmatrix}E-\gamma_1x^*_1,&\gamma_1(x^*_1-x^*_2),&\gamma_1(x^*_2-x^*_3),&0,&0,&0\end{pmatrix}'&
\end{align*}
and $C = \begin{pmatrix}0,&0,&0,&1,&0,&0\end{pmatrix}$
where $\Id_3$ is the $3\times3$ identity matrix, $k, \gamma_1, \gamma_2, E$ are positive physical constants, and
$$J =
\begin{pmatrix}
-1&0&0\\
1&-1&0\\
0&1&-1
\end{pmatrix}.$$
With $G$ a positive physical constant of the system, each control $u^*>0$ leads to exactly one equilibrium state $x^*$
such that $A(0)x^* = \begin{pmatrix}
Eu^*,&0,&0,&0,&0,&G
\end{pmatrix}'$. The matrix $A(0)$ is invertible according to \cite{zitte}.
Again, this system is not uniformly observable.
Indeed, the determinant of the Kalman observability matrix of the pair $(C, A(\bar{u}))$ is $k^3\gamma_2^6(k^2-\gamma_1\gamma_2(\bar{u}+u^*))^3$. Hence, the constant input $\bar{u}\equiv\frac{k^2}{\gamma_1\gamma_2}-u^*$ renders \eqref{E:observation_system} unobservable.

\begin{figure}[h!]
    \centering
    \includegraphics[width =0.6\linewidth]{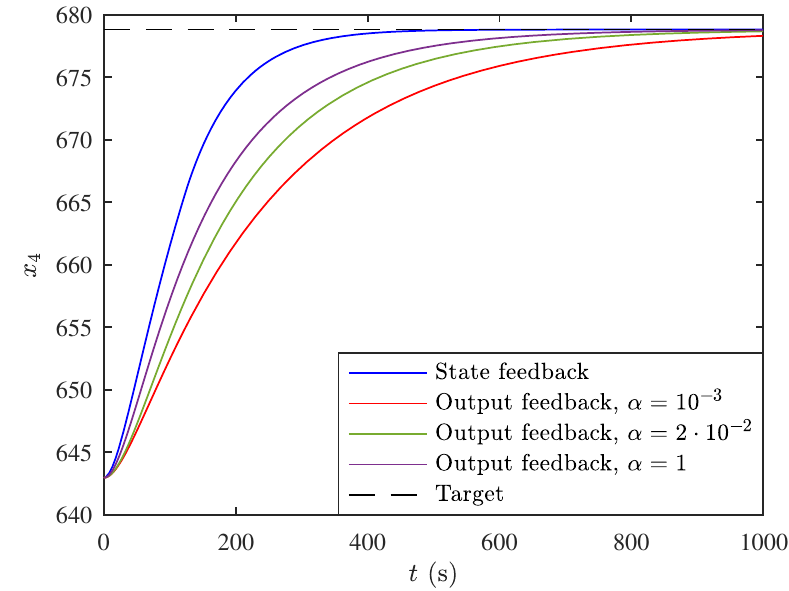}
    \caption{Output enthalpy of the heat exchanger with the state feedback law $\lambda$ and with the corresponding dynamic output feedback law based on the Luenberger observer for different values of $\alpha$.}
    \label{fig:simuheat}
\end{figure}
\begin{figure}[h!]
    \centering
    \includegraphics[width =0.6\linewidth]{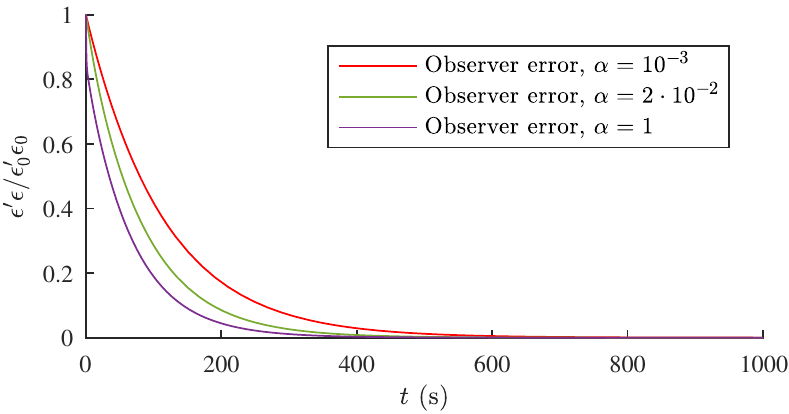}
    \caption{ Evolution of the error between the actual state of the heat exchanger and the observer for different values of $\alpha$.}
    \label{fig:obsheat}
\end{figure}

However, Theorem~\ref{thmain} may apply if Assumptions~\ref{assum:stab}, \ref{assum:diss} and \ref{assum:obs} are satisfied.
Choose $\lambda(\bar{x}) = \mathrm{sat}\left(-\beta b'\bar{x}\right)$, where $\beta>0$ is a tuning parameter and $\mathrm{sat}$ is a saturation function such that $u^* + \lambda$ lies in an interval $(0, u_M)$, which is always possible if $u^*\in(0, u_M)$.
If the pair $(b', A(0))$ is detectable, we apply the LaSalle's invariance principle to the Lyapunov function $x\mapsto x'x$, and get that $x$ converge towards $0$.
Then, Assumption~\ref{assum:diss} is satisfied since $A(\bar{u}) + A(\bar{u})'$ is negative definite when $u^*+\bar{u}>0$ according to the Gershgorin circle theorem.
The pair $(C, A(0))$ is observable, and \emph{a fortiori} detectable, if and only if $u^*\neq\frac{k^2}{\gamma_1\gamma_2}$.
We fix the following parameters, that satisfy all the previous assumptions.

\begin{table}[h!]
\caption{Numerical values for the simulation of the heat exchanger}
\begin{center}
\begin{tabular}{|c|c|c|c|c}
\hline
$k$ & $\gamma_1$ & $\gamma_2$ & $E$\\
$1.20\cdot10^{-2}$ $\mathrm{s}^{-1}$ & $5.06\cdot10^{-1}$ $\mathrm{kg}^{-1}$ & $1.00\cdot10^{-2}$ $\mathrm{s}^{-1}$ & $360$ K\\
\hline
$G$ & $u_M$ & $u^*$ & $\beta$\\
$300$ K & $0.05$ $\mathrm{kg}\cdot\mathrm{s}^{-1}$ & $0.5u_M$ & $1$\\
\hline
\end{tabular}
\end{center}
\end{table}
Set $\xhat(0) = x^*$, and let $x(0)$ be the steady state that corresponds to the constant input $u\equiv0.17u_M$.
Then Theorem~\ref{thmain} build a dynamic output feedback based on $\lambda$ and a Luenberger observer.
In Figure~\ref{fig:simuheat}, we plot the evolution of the output $x_4$ (that we intend to stabilize as in \cite{zitte}) for the state feedback law $\lambda$ and for the dynamic output feedback based on the observer for $\alpha = 10^{-3}$, $\alpha=2\cdot10^{-2}$ and $\alpha=1$. The error between the state and the observer is given in Figure \ref{fig:obsheat} for the same values of $\alpha$.
As in Example~\ref{ex:cuk}, the convergence of the observer to the state of the system is faster when $\alpha$ is larger, and then the stabilization of the state with dynamic output feedback gets closer to the one obtained by state feedback.
\end{exmp}

\section{Conclusion}

We have shown that for dissipative systems, target detectability is a sufficient condition to move from a locally asymptotically stabilizing state feedback to a semi-globally asymptotically stabilizing dynamic output feedback.
In particular, we are able to remove the classical uniform observability assumption.
Therefore, this work lays foundations for more general strategies in output feedback stabilization of non-uniformly observable systems, notably via embeddings into dissipative systems.

\section{Addendum: Global stabilization}

By allowing the observer gain to depend on the output of the system, it is actually possible to obtain a global version of Theorem~\ref{thmain} instead of a semi-global one.

\begin{thm}\label{thmain_glob}
If Assumptions~\ref{assum:stab}, \ref{assum:diss} and \ref{assum:obs} hold,
then there exists a locally Lipschitz function $\alpha: \R^n\times\R^p\to\R_+$ such that $(0, 0)$ is an asymptotically stable equilibrium point with basin of attraction containing $\D\times\R^n$ of
\begin{equation}\label{E:observer_system_glob}
\left\{
\begin{aligned}
&\dot{\xhat}= A(\lambda(\xhat))\xhat
+ B(\lambda(\xhat))
- \alpha(\xhat, C\eps) P^{-1}C'C\eps\\
&\dot{\eps}= \left(A(\lambda(\xhat)) - \alpha(\xhat, C\eps) P^{-1}C'C\right)\eps.
\end{aligned}
\right.
\end{equation}
\end{thm}

\begin{proof}
Let $\D$ and $f$ be as in Assumption~\ref{assum:stab}, and $P$ as in Assumption~\ref{assum:diss}.
For all $(\xhat, y)\in\R^n\times\R^p$, let $k(\xhat, y) = -\alpha(\xhat, y) P^{-1}C'y$.
According to the converse Lyapunov theorem (see e.g. \cite{teel2000smooth}), there exists a proper function $W\in C^\infty(\D,\R_+)$ such that $W(0)=0$ and
\begin{equation}\label{eq_Lyap}
\frac{\partial W}{\partial x}(x)f(x,\lambda(x)) \le -W(x),\quad \forall x\in\D.
\end{equation}
For all $r>0$, set $D(r) = \{x\in\R^n\mid V(x) \le r\} $ which is a compact subset of $\D$.
Let $\alpha:\R^n\times\R^p\to\R_+$ be the function defined for all $(\xhat, y)\in\R^n\times\R^p$ by
\begin{equation}
\alpha(\xhat, y) = 
-\frac{\max\{W(\hat x), 1\}}{2\left(1 + \left|\frac{\partial W}{\partial x}(\hat x)\right|\right)
\left(1+
\left|P^{-1}C'y\right|\right)}.
\end{equation}
Note that $\alpha$ is locally Lipschitz and $\alpha(\hat x, y)>0$ for all $(\xhat, y)\in\R^n\times\R^p$. Also, it yields
\begin{equation}
|k(\xhat, y)| \le \frac{\max\{W(\hat x),1\}}{2\left(1 + \left|\frac{\partial W}{\partial x}(\hat x)\right|\right)},\quad \forall(\xhat, y)\in\R^n\times\R^p.
\end{equation}

\noindent\textbf{Step 1: local asymptotic stability
}
Since $\alpha(0, 0)>0$, the linearization of \eqref{E:observer_system_glob} at $(0, 0)$ may be rewritten in the form of \eqref{E:linearized_observer_system} and the proof of local asymptotic stability follows from Section~\ref{secloc}.\\

\noindent\textbf{Step 2: all trajectories are bounded.
}
Consider the function $V:\eps \mapsto \eps'P\eps$.
For all initial conditions $(\hat x_0, \eps_0)\in\D\times\R^n$, the 
solution of the closed-loop system \eqref{E:observer_system_glob} denoted $(\xhat(\cdot), \eps(\cdot))$ satisfies
\begin{align}
    \frac{\diff V(\eps)}{\diff t}
    &= \eps'P\dot\eps + \dot\eps'P\eps\nonumber\\
    &= \eps'\left(P A(\lambda(\xhat)) + A(\lambda(\xhat))'P\right) \eps - 2\alpha \eps'C'C\eps\nonumber\\
    &\leq - 2\alpha |C\eps|^2\tag{by Assumption~\ref{assum:diss}}\\
    &\leq 0.\nonumber
\end{align}
Hence, $\eps$ remains in a compact set.

Moreover, for all $(\hat x, y)\in\R^n\times\R^p$,
\begin{align*}
\frac{\partial W}{\partial x}(\xhat)[f(\hat x,\lambda(\hat x)) + k(\hat x, y)]
&\le -W(\hat x) + \frac{\partial W}{\partial x}(\xhat)k(\hat x, y)\\
&\le -W(\hat x) + \left|\frac{\partial W}{\partial x}(\xhat)\right| |k(\hat x, y)|\\
&\le  -W(\hat x) + \left|\frac{\partial W}{\partial x}(\xhat)\right| \frac{\max\{W(\hat x),1\}}{2\left(1 + \left|\frac{\partial W}{\partial x}(\xhat)\right|\right)}
\\
&\le  -W(\hat x) + \frac{1}{2}\max\{W(\hat x),1\}.
\end{align*}
Hence, if $\xhat\in \D\setminus D(1)$,
\begin{equation}
    \frac{\partial W}{\partial x}(\xhat)(f(\hat x,\lambda(\hat x)) + k(\xhat, y))\le -\frac{1}{2} W(\hat x).
\end{equation}
Thus
\begin{equation*}
W(\hat x)\le \max\{ W(\hat x_0), 1\},
\end{equation*}
In other words, $\hat x$ remains in $D(1)\cup D(W(\hat x_0))$ which is a compact subset of $\D$.
Thus, solutions of \eqref{E:observer_system_glob} are complete in positive time.\\

\noindent\textbf{Step 3: all trajectories converge to $0$.}
Since all trajectories of \eqref{E:observer_system_glob} are bounded from Step 2, the proof of convergence is identical to Section~\ref{secconv}.

\end{proof}

\section*{Acknowledgments}
The authors would like to thank Bertrand Zitte, Boussad Hamroun, Fran\c{c}oise Couenne and
Daniele Astolfi (authors of \cite{zitte}) for many fruitful discussions about physical systems satisfying the assumptions under consideration.\\

This research was funded by the French Grant ANR ODISSE (ANR-19-CE48-0004-01).
\bibliographystyle{abbrv}
\bibliography{references}

\end{document}